\documentclass[11pt,twoside]{amsart}
\usepackage{amsthm,amsfonts,amssymb,amscd,euscript}
\usepackage[dvips]{graphicx}

\usepackage{color}
\usepackage{listings}
\usepackage{graphics}
\usepackage{tikz}
\input{xy}
\xyoption{all}

\usepackage[colorlinks=true, pdfstartview=FitV,linkcolor=black,citecolor=black,urlcolor=black]{hyperref}

\newenvironment{red}{\relax\color{red}}{\relax}
\newenvironment{blue}{\relax\color{blue}}{\hspace*{.5ex}\relax}
\newcommand{\ber}{\begin{red}}
\newcommand{\er}{\end{red}}
\newcommand{\beb}{\begin{blue}}
\newcommand{\eb}{\end{blue}}

\theoremstyle{plain}
\newtheorem{thm}{Theorem}[section]
\newtheorem{lem}[thm]{Lemma}
\newtheorem{prop}[thm]{Proposition}
\newtheorem{cor}[thm]{Corollary}
\newtheorem{conj}[thm]{Conjecture}

\theoremstyle{definition}
\newtheorem{definition}[thm]{Definition}
\newtheorem{example}[thm]{Example}
\newtheorem{remark}[thm]{Remark}
\numberwithin{equation}{section} \numberwithin{figure}{section}
\numberwithin{table}{section}

\begin{document}
\title{Rigid reflections and Kac--Moody algebras}
\author[K.-H. Lee]{Kyu-Hwan Lee$^{\star}$}
\thanks{$^{\star}$This work was partially supported by a grant from the Simons Foundation (\#318706).}
\address{Department of
Mathematics, University of Connecticut, Storrs, CT 06269, U.S.A.}
\email{khlee@math.uconn.edu}

\author[K. Lee]{Kyungyong Lee$^{\dagger}$}
\thanks{$^{\dagger}$This work was partially supported by the University of Nebraska--Lincoln and Korea Institute for Advanced Study.}
\address{Department of Mathematics, University of Nebraska--Lincoln, Lincoln, NE 68588, U.S.A.,
and Korea Institute for Advanced Study, Seoul 02455, Republic of Korea}
\email{klee24@unl.edu; klee1@kias.re.kr}

\begin{abstract}
Given any Coxeter group, we define rigid reflections and rigid roots 
using non-self-intersecting curves on a Riemann surface with labeled curves. When the Coxeter group arises from an acyclic quiver,  they are related to the rigid representations of the quiver. 
For a family of rank $3$ Coxeter groups, we show that there is a surjective map from  the set of reduced positive roots of a rank $2$ Kac--Moody algebra onto the set of rigid reflections. We conjecture that this map is bijective.
\end{abstract}

\maketitle

\section{Introduction}

Let $Q$ be an acyclic quiver of rank $n$, i.e. a quiver with $n$ vertices and without oriented cycles, and $\mathrm{mod}(Q)$ be the category of finite dimensional representations of $Q$, or equivalently, the category of finite dimensional modules over the path algebra of $Q$. Among the objects in $\mathrm{mod}(Q)$, the indecomposable representations $M$ with $\mathrm{Ext}^1(M,M)=0$ are called {\em rigid} and their dimension vectors are called {\em real Schur roots}. They play prominent roles in understanding the category $\mathrm{mod}(Q)$. Moreover, the real Schur roots form an important subset of the set of positive real roots of the Kac--Moody algebra $\mathfrak g(Q)$ associated to $Q$. 

In an attempt to establish a form of Homological Mirror Symmetry \cite{Kon}, we proposed in a previous paper \cite{LL}  a correspondence between rigid representations in $\mathrm{mod}(Q)$ and the set of certain non-self-intersecting
curves on a Riemann surface $\Sigma$ with $n$ labeled curves. The conjecture is now proven by Felikson and Tumarkin \cite{FT} for {\em 2-complete} quivers $Q$ where, by definition, every pair of vertices in $Q$ is connected by more than two edges. However it is wide open for general acyclic quivers. 

The conjectural correspondence factors through a family of reflections in the Weyl group of $\mathfrak g(Q)$ to relate non-self-intersecting curves in $\Sigma$ with real Schur roots. Since  reflections make sense 
for any Coxeter groups, one can consider such a family of reflections in an arbitrary Coxeter group. Indeed, let $W$ be a Coxeter group with $n$ ordered generators:
$$
W=\langle s_1,s_2,...,s_n \ : \   s_1^2=\cdots=s_n^2=e, \ (s_is_j)^{m_{ij}}=e \ (i \neq j) \rangle,
$$ where $m_{ij}\in\{2,3,4,...\} \cup \{ \infty \}$. 
Then we define {\em rigid reflections} in $W$ using non-self-intersecting curves on the Riemann surface $\Sigma$ and {\em rigid roots} to be the associated positive roots to the rigid reflections in the set of positive roots of $W$. (See Definitions \ref{def-rr} and \ref{def-rroot}.)  These definitions are in line with the conjectural correspondence.  In particular, when $m_{ij}=\infty$ for all $i\neq j$, the set of rigid reflections is in bijection with the set of rigid representations of any 2-complete acyclic quiver $Q$ of rank $n$.

This paper is concerned with an unexpected, surprising  phenomenon that the rigid roots of $W$ are parametrized by the positive roots of a seemingly unrelated Kac--Moody algebra $\mathcal H$. 
This phenomenon seems true for a wide range of Coxeter groups $W$, and we will show this for  a family of rank 3 Coxeter groups in this paper.  

To be precise, fix a positive integer $m\geq 2$ and consider the following Coxeter group
 $$W(m)=\langle s_1,s_2,s_3 \ : \   s_1^2=s_2^2=s_3^2=(s_1s_2)^m=(s_2s_3)^m=e \rangle.
 $$
Here we put $m_{13}=m_{31}=\infty$ as the usual convention.
Let  $\mathcal H(m)$ be the rank 2 Kac--Moody algebra associated with the Cartan matrix ${\tiny \begin{pmatrix}  2&-m\\-m&2\end{pmatrix}}$. We denote an element of the root lattice of $\mathcal H(m)$ by $[a,b]$, $a, b \in \mathbb Z$, where $[1,0]$ and $[0,1]$ are the positive simple roots. 
A root $[a,b]$ of $\mathcal H(m)$ is called {\em reduced} if $\gcd(a,b)=1$ and $ab\neq 0$. A reduced root determines a non-self-intersecting curve $\eta$ on the torus $\Sigma$ with triangulation by three labeled curves. 
Then we define a function, $[a,b] \mapsto s([a,b]) \in W$, by reading off the labels of the intersection points of $\eta$ with the labeled curves on $\Sigma$, and make the following conjecture.

\begin{conj} \label{conj}
For $m \ge 2$, the function, $[a,b] \mapsto s([a,b])$, is a bijection from the set of reduced positive roots of $\mathcal H(m)$ to  the set of rigid reflections of $W(m)$. 
\end{conj}

The case $m=2$  will be verified in Example \ref{exm-2}, and the case $m=3$ will be established in a forthcoming paper \cite{LL1} where  
mutations of quivers and cluster variables will be exploited.
As the main result of this paper, we prove

\begin{thm} \label{thm-in}
For $m \ge 2$, the function in Conjecture \ref{conj} is a surjection.
\end{thm}

Our proof of Theorem \ref{thm-in} shows that the Weyl group  of $\mathcal H(m)$, which is isomorphic to the infinite dihedral group, governs the symmetries of the set of rigid reflections of $W(m)$, and utilizes these symmetries to make an induction argument work on the values of the square norm of $[a,b]$. It is intriguing that such a nice structure dwells in the set of rigid reflections.

The organization of this paper is as follows. In Section \ref{rigid}, we define rigid reflections and rigid roots and provide examples. After introducing notations for rank $2$ Kac--Moody algebras $\mathcal H(m)$, we state the main theorem and illustrate with examples. In particular, the case $m=2$ is completely described. Section \ref{proof} is devoted to a proof of the main theorem.
We first  establish several lemmas, and  a main step is achieved in Proposition \ref{pro-pr} whose proof is an inductive algorithm for, given $[a,b]$ in the positive root lattice, how to find a reduced positive root $[a_0,b_0]$ of $\mathcal H (m)$ with the same rigid reflections, i.e, $s([a_0,b_0])=s([a,b])$. Then Lemma \ref{lem-page} shows that it is enough to consider the positive root lattice, which completes the proof.

\section{Rigid reflections and the main theorem} \label{rigid}
As in the introduction, let $$
W=\langle s_1,s_2,...,s_n \ : \   s_1^2=\cdots=s_n^2=e, \ (s_is_j)^{m_{ij}}=e \rangle
$$ 
be a Coxeter group with $m_{ij}\in\{2,3,4,...\} \cup \{ \infty \}$. In this section, after the rigid reflections in $W$ and rigid roots in the root system of $W$ are defined, the main theorem of this paper will be stated.

\medskip

To begin with, we need a Riemann surface  $\Sigma$ equipped with $n$ labeled curves as below.
Let $G_1$ and $G_2$ be two identical copies of a regular $n$-gon. Label the edges of each of the two  $n$-gons
 by $T_{1}, T_{2}, \dots , T_{n}$ counter-clockwise. On $G_i$ $(i=1,2)$, let $L_i$ be the line segment from the center of $G_i$ to the common endpoint of  $T_{n}$ and $T_{1}$. Later,  these line segments will only be used  to designate the end points of admissible curves and will not be used elsewhere.   Fix the orientation of every edge of $G_1$ (resp.  $G_2$) to be 
 counter-clockwise (resp. clockwise) as in the following picture. 
 
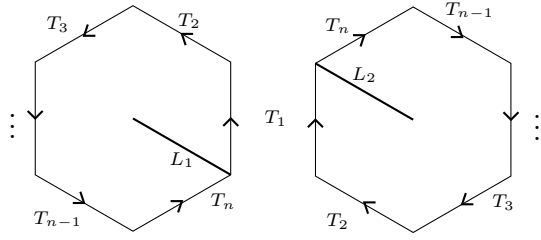
\begin{center}
 \begin{tikzpicture}[scale=0.5]
\node at (2.4,-2.2){\tiny{$T_n$}};
\node at (1.5,2.6){\tiny{$T_2$}};
\node at (3.8,0){\tiny{$T_1$}};
\node at (-2.0,-2.7){\tiny{$T_{n-1}$}};
\node at (-2.0,2.5){\tiny{$T_3$}};
\node at (-3.2,0){\vdots};
\draw (0,0) +(30:3cm) -- +(90:3cm) -- +(150:3cm) -- +(210:3cm) --
+(270:3cm) -- +(330:3cm) -- cycle;
\draw [thick] (2.4,-0.2) -- (2.6,0)--(2.8,-0.2);
\draw [thick] (1.4,1.95) -- (1.3,2.25)--(1.6,2.25);  
\draw [thick] (-1.0,2.2) -- (-1.3,2.2)--(-1.2,2.5);  
\draw [thick] (-2.4,0.2) -- (-2.6,0)--(-2.8,0.2);   
\draw [thick] (1.0,-2.2) -- (1.3,-2.2)--(1.2,-2.5);  
\draw [thick] (-1.4,-1.95) -- (-1.3,-2.25)--(-1.6,-2.25);  
\draw [thick] (0,0)--(2.6,-1.5);  
\node at (1.3,-1.1){\tiny{$L_1$}};
\end{tikzpicture}
\begin{tikzpicture}[scale=0.5]
\node at (-1.3,1.2){\tiny{$L_2$}};
\draw [thick] (0,0)--(-2.6,1.5);  
\node at (2.4,-2.2){\tiny{$T_3$}};
\node at (1.5,2.9){\tiny{$T_{n-1}$}};
\node at (3.3,0){\vdots};
\node at (-2.0,-2.7){\tiny{$T_2$}};
\node at (-2.0,2.5){\tiny{$T_n$}};
\draw (0,0) +(30:3cm) -- +(90:3cm) -- +(150:3cm) -- +(210:3cm) --
+(270:3cm) -- +(330:3cm) -- cycle;
\draw [thick] (-2.4,-0.2) -- (-2.6,0)--(-2.8,-0.2);
\draw [thick] (-1.4,1.95) -- (-1.3,2.25)--(-1.6,2.25);  
\draw [thick] (1.0,2.2) -- (1.3,2.2)--(1.2,2.5);  
\draw [thick] (2.4,0.2) -- (2.6,0)--(2.8,0.2);   
\draw [thick] (-1.0,-2.2) -- (-1.3,-2.2)--(-1.2,-2.5);  
\draw [thick] (1.4,-1.95) -- (1.3,-2.25)--(1.6,-2.25);  
\end{tikzpicture}
 \end{center}

 Let $\Sigma$ be the Riemann surface of genus $\lfloor \frac{n-1}{2}\rfloor$
obtained by gluing together the two $n$-gons with all the edges of the same label identified according 
to their orientations.  The edges of the $n$-gons become $n$ different curves in $\Sigma$. If $n$ is odd, all the vertices of the two $n$-gons 
are identified to become one point in $\Sigma$ and the curves obtained from the edges become loops. If $n$ is even, two distinct
 vertices are shared by all curves. Let $\mathcal{T}={T}_1\cup\cdots{T}_n\subset \Sigma$, and $V$ be the set of the vertex (or vertices) on $\mathcal{T}$.  

 Let $\mathfrak W$ be the set of words from the alphabet $\{1,2,...,n\}$, and let $\mathfrak R\subset\mathfrak W$ be the subset of words $\mathfrak w=i_1i_2 \cdots i_k$ such that $k$ is an odd integer and $i_{j}=i_{k+1-j}$ for all $j\in\{1,...,k\}$, in other words, $s_{i_1}s_{i_2} \cdots s_{i_k}$ is a reflection in $W$. For  $\mathfrak{w}=i_1i_2 \cdots i_k\in \mathfrak W$, denote $s_{i_1}...s_{i_k}\in W$ by $s(\mathfrak{w})$.

\begin{definition} 
An \emph{admissible} curve is a continuous function $\eta:[0,1]\longrightarrow \Sigma$ such that

1) $\eta(x)\in V$ if and only if  $x\in\{0,1\}$;

2) $\eta$ starts and ends at the common end point of $T_1$ and $T_n$. More precisely, there exists $\epsilon>0$ such that $\eta([0,\epsilon])\subset L_1$ and $\eta([1-\epsilon,1])\subset L_2$;

3) if $\eta(x)\in \mathcal{T}\setminus V$ then $\eta([x-\epsilon,x+\epsilon])$ meets $\mathcal{T}$ transversally for sufficiently small $\epsilon>0$.
\end{definition}

If $\eta$ is admissible, then we obtain $\upsilon(\eta):={i_1}\cdots {i_k}\in \mathfrak W$  given by 
$$\{x\in(0,1) \ : \ \eta(x)\in \mathcal{T}\}=\{x_1<\cdots<x_k\}\quad \text{ and }\quad \eta(x_\ell)\in T_{i_\ell}\text{ for }\ell\in\{1,...,k\}.$$ 
Conversely, note that for every $\mathfrak w\in \mathfrak W$, there is an admissible curve $\eta$ with $\upsilon(\eta)=\mathfrak{w}$. Hence, every element in $W$   can be represented by some admissible curve(s). For brevity, let $s(\eta):=s(\upsilon(\eta))$.

\begin{definition} \label{def-rr}
An element $s_{i_1}s_{i_2} \cdots s_{i_k}$ in $W$ is called a \emph{rigid reflection} if there exists a non-self-crossing admissible curve $\eta$ with  $\upsilon(\eta)=i_1...i_k\in \mathfrak R$. 
\end{definition}

\begin{example} \label{exa-1}
Let $n=3$, and $W = \langle s_1, s_2, s_3 \ : \ s_1^2=s_2^2=s_3^2 =e \rangle $, i.e., $m_{ij} = \infty$ for $i \neq j$. Consider the universal cover of $\Sigma$ and a curve $\eta$ as in the following picture. 
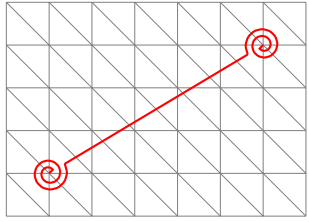
\begin{center}
\begin{tikzpicture}[scale=0.2mm]
\draw [help lines] (0,0) grid (7,5);
\draw [help lines] (0,1)--(1,0);
\draw [help lines] (0,2)--(2,0);
\draw [help lines] (0,3)--(3,0);
\draw [help lines] (0,4)--(4,0);
\draw [help lines] (0,5)--(5,0);
\draw [help lines] (1,5)--(6,0);
\draw [help lines] (2,5)--(7,0);
\draw [help lines] (3,5)--(7,1);
\draw [help lines] (4,5)--(7,2);
\draw [help lines] (5,5)--(7,3);
\draw [help lines] (6,5)--(7,4);
\draw [thick,red] (1,1)--(1.1,1.1);  
\draw [thick,red] (6,4)--(5.9,3.9);  
\draw [thick,red] (1.36,1.22)--(6-0.36,4-0.22);  
\draw[thick,red,scale=0.5,domain=0.75:13.12,smooth,variable=\t]
plot ({(0.2+0.05*\t)*cos(\t r)+2},{(0.2+0.05*\t)*sin(\t r)+2});
\draw[thick,red,scale=0.5,domain=3.89:16.26,smooth,variable=\t]
plot ({(0.03+0.05*\t)*cos(\t r)+12},{(0.03+0.05*\t)*sin(\t r)+8});
\end{tikzpicture}
\end{center}
Here each horizontal line segment represents $T_1$, vertical $T_3$, and diagonal $T_2$.
One sees that $\eta$ has no self-intersection in $\Sigma$. Thus we obtain the corresponding rigid reflection  
\[ s(\eta) = (s_3s_2s_1)^4s_2s_3s_2s_1s_2s_3s_2s_3s_2s_1s_2s_3s_2(s_1s_2s_3)^4 . \]

On the other hand, the reflection $s_2s_3s_1s_3s_2$ comes from the following curve $\eta'$ which has a self-intersection.  The picture on the right shows several copies of $\eta'$ on the universal cover.

 
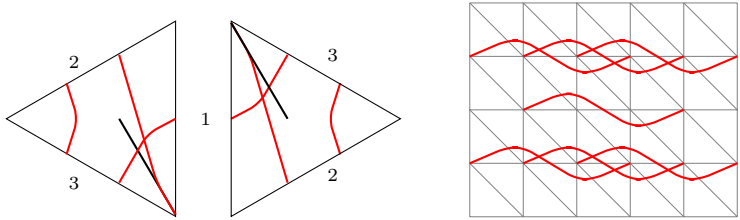
\begin{center}
 \begin{tikzpicture}[scale=0.5]
 \node at (2.3,0){\tiny{$1$}};
\node at (-1.2,-1.7){\tiny{$3$}};
\node at (-1.2,1.5){\tiny{$2$}};
\draw (0,0) +(60:3cm) -- +(180:3cm) -- +(300:3cm) -- cycle;
\draw [thick] (0,0) -- (1.5,-2.55); 
\draw [red, thick] (0,1.7)
[rounded corners] --(1,-1.7) -- (1.5,-2.55); 
\draw [red, thick] (-1.4,0.95)
[rounded corners] --(-1.1,0) -- (-1.4,-0.95); 
\draw [red, thick] (0,-1.7)
[rounded corners] --(0.75,-0.4) -- (1.5,0); 
\end{tikzpicture}
\begin{tikzpicture}[scale=0.5]
\draw [red, thick] (0,-1.7)
[rounded corners] --(-1,1.7) -- (-1.5,2.55); 
\draw [red, thick] (1.4,-0.95)
[rounded corners] --(1.1,0) -- (1.4,0.95); 
\draw [red, thick] (0,1.7)
[rounded corners] --(-0.75,0.4) -- (-1.5,0); 
\draw [thick] (0,0) -- (-1.5,2.55); 
\node at (1.2,1.7){\tiny{$3$}};
\node at (1.2,-1.5){\tiny{$2$}};
\draw (0,0) +(0:3cm) -- +(120:3cm) -- +(240:3cm) -- cycle;
\end{tikzpicture}
\qquad
\begin{tikzpicture}[scale=0.25mm]
\draw [help lines] (0,0) grid (5,4);
\draw [help lines] (0,1)--(1,0);
\draw [help lines] (0,2)--(2,0);
\draw [help lines] (0,3)--(3,0);
\draw [help lines] (0,4)--(4,0);
\draw [help lines] (1,4)--(5,0);
\draw [help lines] (2,4)--(5,1);
\draw [help lines] (3,4)--(5,2);
\draw [help lines] (4,4)--(5,3);
\draw [red, thick] (2-1,2) 
[rounded corners] -- (2.7-1,2.3) -- (3-1,2.3) -- (3.5-1,2) -- (4-1,1.7) -- (4.3-1,1.7)--(5-1,2);
\draw [red, thick] (2,2-1) 
[rounded corners] -- (2.7,2.3-1) -- (3,2.3-1) -- (3.5,2-1) -- (4,1.7-1) -- (4.3,1.7-1)--(5,2-1);
\draw [red, thick] (2-1,2-1) 
[rounded corners] -- (2.7-1,2.3-1) -- (3-1,2.3-1) -- (3.5-1,2-1) -- (4-1,1.7-1) -- (4.3-1,1.7-1)--(5-1,2-1);
\draw [red, thick] (2-2,2-1) 
[rounded corners] -- (2.7-2,2.3-1) -- (3-2,2.3-1) -- (3.5-2,2-1) -- (4-2,1.7-1) -- (4.3-2,1.7-1)--(5-2,2-1);
\draw [red, thick] (2,2+1) 
[rounded corners] -- (2.7,2.3+1) -- (3,2.3+1) -- (3.5,2+1) -- (4,1.7+1) -- (4.3,1.7+1)--(5,2+1);
\draw [red, thick] (2-1,2+1) 
[rounded corners] -- (2.7-1,2.3+1) -- (3-1,2.3+1) -- (3.5-1,2+1) -- (4-1,1.7+1) -- (4.3-1,1.7+1)--(5-1,2+1);
\draw [red, thick] (2-2,2+1) 
[rounded corners] -- (2.7-2,2.3+1) -- (3-2,2.3+1) -- (3.5-2,2+1) -- (4-2,1.7+1) -- (4.3-2,1.7+1)--(5-2,2+1);
\end{tikzpicture}
 \end{center}
Consequently, the reflection $s(\eta')=s_2s_3s_1s_3s_2$ is {\em not} rigid.

\end{example}

\begin{example}
Let $n=8$, and we have a rigid reflection  $$(s_8s_7\cdots s_2s_1)^5(s_8s_7\cdots s_2)s_1(s_2\cdots s_7s_8)(s_1s_2\cdots s_7s_8)^5,$$ which corresponds to the following non-self-intersecting curve on $\Sigma$. 

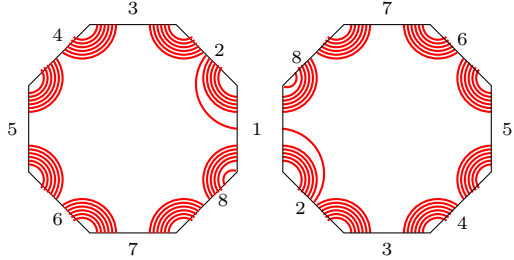
\begin{center}
\begin{tikzpicture}[scale=0.5]
\node at (0,-3.2){\tiny{$7$}};
\node at (2.4,-1.9){\tiny{$8$}};
\node at (2.3,2.1){\tiny{$2$}};
\node at (3.3,0){\tiny{$1$}};
\node at (0,3.2){\tiny{$3$}};
\node at (-2.0,-2.4){\tiny{$6$}};
\node at (-2.0,2.5){\tiny{$4$}};
\node at (-3.2,0){\tiny{$5$}};
\draw (0,0) +(22.5:3cm) -- +(67.5:3cm) -- +(112.5:3cm) -- +(157.5:3cm) --
+(202.5:3cm) -- +(247.5:3cm) -- +(292.5:3cm) -- +(337.5:3cm) -- cycle;
\draw [red, thick] (0.95,2.77) arc (180:305:0.4cm);
\draw [red, thick] (0.85,2.77) arc (180:305:0.5cm);
\draw [red, thick] (0.75,2.77) arc (180:305:0.6cm);
\draw [red, thick] (0.65,2.77) arc (180:305:0.7cm);
\draw [red, thick] (0.55,2.77) arc (180:305:0.8cm);
\draw [red, thick] (0.45,2.77) arc (180:305:0.9cm);
\draw [red, thick] (-0.95,2.77) arc (360:235:0.4cm);
\draw [red, thick] (-0.85,2.77) arc (360:235:0.5cm);
\draw [red, thick] (-0.75,2.77) arc (360:235:0.6cm);
\draw [red, thick] (-0.65,2.77) arc (360:235:0.7cm);
\draw [red, thick] (-0.55,2.77) arc (360:235:0.8cm);
\draw [red, thick] (-0.45,2.77) arc (360:235:0.9cm);
\draw [red, thick] (0.95,-2.77) arc (180:55:0.4cm);
\draw [red, thick] (0.85,-2.77) arc (180:55:0.5cm);
\draw [red, thick] (0.75,-2.77) arc (180:55:0.6cm);
\draw [red, thick] (0.65,-2.77) arc (180:55:0.7cm);
\draw [red, thick] (0.55,-2.77) arc (180:55:0.8cm);
\draw [red, thick] (0.45,-2.77) arc (180:55:0.9cm);
\draw [red, thick] (-0.95,-2.77) arc (0:125:0.4cm);
\draw [red, thick] (-0.85,-2.77) arc (0:125:0.5cm);
\draw [red, thick] (-0.75,-2.77) arc (0:125:0.6cm);
\draw [red, thick] (-0.65,-2.77) arc (0:125:0.7cm);
\draw [red, thick] (-0.55,-2.77) arc (0:125:0.8cm);
\draw [red, thick] (-0.45,-2.77) arc (0:125:0.9cm);
\draw [red, thick] (-2.77,0.95) arc (270:395:0.4cm);
\draw [red, thick] (-2.77,0.85) arc (270:395:0.5cm);
\draw [red, thick] (-2.77,0.75) arc (270:395:0.6cm);
\draw [red, thick] (-2.77,0.65) arc (270:395:0.7cm);
\draw [red, thick] (-2.77,0.55) arc (270:395:0.8cm);
\draw [red, thick] (-2.77,0.45) arc (270:395:0.9cm);
\draw [red, thick] (-2.77,-0.95) arc (90:-35:0.4cm);
\draw [red, thick] (-2.77,-0.85) arc (90:-35:0.5cm);
\draw [red, thick] (-2.77,-0.75) arc (90:-35:0.6cm);
\draw [red, thick] (-2.77,-0.65) arc (90:-35:0.7cm);
\draw [red, thick] (-2.77,-0.55) arc (90:-35:0.8cm);
\draw [red, thick] (-2.77,-0.45) arc (90:-35:0.9cm);
\draw [red, thick] (2.45,-1.5) arc (215:70:0.25cm);    
\draw [red, thick] (2.77,-0.85) arc (90:215:0.5cm);
\draw [red, thick] (2.77,-0.75) arc (90:215:0.6cm);
\draw [red, thick] (2.77,-0.65) arc (90:215:0.7cm);
\draw [red, thick] (2.77,-0.55) arc (90:215:0.8cm);
\draw [red, thick] (2.77,-0.45) arc (90:215:0.9cm);
\draw [red, thick] (2.77,0.85) arc (270:145:0.5cm);
\draw [red, thick] (2.77,0.75) arc (270:145:0.6cm);
\draw [red, thick] (2.77,0.65) arc (270:145:0.7cm);
\draw [red, thick] (2.77,0.55) arc (270:145:0.8cm);
\draw [red, thick] (2.77,0.45) arc (270:145:0.9cm);
\draw [red, thick] (2.77,0) arc (265:141:1.2cm);
\end{tikzpicture}
\begin{tikzpicture}[scale=0.5]
\draw [red, thick] (-2.77,0) arc (85:-39:1.2cm);
\draw [red, thick] (-2.45,1.5) arc (35:-110:0.25cm);   
\draw [red, thick] (-2.77,0.85) arc (270:395:0.5cm);
\draw [red, thick] (-2.77,0.75) arc (270:395:0.6cm);
\draw [red, thick] (-2.77,0.65) arc (270:395:0.7cm);
\draw [red, thick] (-2.77,0.55) arc (270:395:0.8cm);
\draw [red, thick] (-2.77,0.45) arc (270:395:0.9cm);
\draw [red, thick] (-2.77,-0.85) arc (90:-35:0.5cm);
\draw [red, thick] (-2.77,-0.75) arc (90:-35:0.6cm);
\draw [red, thick] (-2.77,-0.65) arc (90:-35:0.7cm);
\draw [red, thick] (-2.77,-0.55) arc (90:-35:0.8cm);
\draw [red, thick] (-2.77,-0.45) arc (90:-35:0.9cm);
\draw [red, thick] (2.77,-0.95) arc (90:215:0.4cm);
\draw [red, thick] (2.77,-0.85) arc (90:215:0.5cm);
\draw [red, thick] (2.77,-0.75) arc (90:215:0.6cm);
\draw [red, thick] (2.77,-0.65) arc (90:215:0.7cm);
\draw [red, thick] (2.77,-0.55) arc (90:215:0.8cm);
\draw [red, thick] (2.77,-0.45) arc (90:215:0.9cm);
\draw [red, thick] (2.77,0.95) arc (270:145:0.4cm);
\draw [red, thick] (2.77,0.85) arc (270:145:0.5cm);
\draw [red, thick] (2.77,0.75) arc (270:145:0.6cm);
\draw [red, thick] (2.77,0.65) arc (270:145:0.7cm);
\draw [red, thick] (2.77,0.55) arc (270:145:0.8cm);
\draw [red, thick] (2.77,0.45) arc (270:145:0.9cm);
\draw [red, thick] (0.95,2.77) arc (180:305:0.4cm);
\draw [red, thick] (0.85,2.77) arc (180:305:0.5cm);
\draw [red, thick] (0.75,2.77) arc (180:305:0.6cm);
\draw [red, thick] (0.65,2.77) arc (180:305:0.7cm);
\draw [red, thick] (0.55,2.77) arc (180:305:0.8cm);
\draw [red, thick] (0.45,2.77) arc (180:305:0.9cm);
\draw [red, thick] (-0.95,2.77) arc (360:235:0.4cm);
\draw [red, thick] (-0.85,2.77) arc (360:235:0.5cm);
\draw [red, thick] (-0.75,2.77) arc (360:235:0.6cm);
\draw [red, thick] (-0.65,2.77) arc (360:235:0.7cm);
\draw [red, thick] (-0.55,2.77) arc (360:235:0.8cm);
\draw [red, thick] (-0.45,2.77) arc (360:235:0.9cm);
\draw [red, thick] (0.95,-2.77) arc (180:55:0.4cm);
\draw [red, thick] (0.85,-2.77) arc (180:55:0.5cm);
\draw [red, thick] (0.75,-2.77) arc (180:55:0.6cm);
\draw [red, thick] (0.65,-2.77) arc (180:55:0.7cm);
\draw [red, thick] (0.55,-2.77) arc (180:55:0.8cm);
\draw [red, thick] (0.45,-2.77) arc (180:55:0.9cm);
\draw [red, thick] (-0.95,-2.77) arc (0:125:0.4cm);
\draw [red, thick] (-0.85,-2.77) arc (0:125:0.5cm);
\draw [red, thick] (-0.75,-2.77) arc (0:125:0.6cm);
\draw [red, thick] (-0.65,-2.77) arc (0:125:0.7cm);
\draw [red, thick] (-0.55,-2.77) arc (0:125:0.8cm);
\draw [red, thick] (-0.45,-2.77) arc (0:125:0.9cm);
\node at (0,3.2){\tiny{$7$}};
\node at (-2.4,1.9){\tiny{$8$}};
\node at (-2.3,-2.1){\tiny{$2$}};
\node at (0,-3.2){\tiny{$3$}};
\node at (2.0,2.4){\tiny{$6$}};
\node at (2.0,-2.5){\tiny{$4$}};
\node at (3.2,0){\tiny{$5$}};
\draw (0,0) +(22.5:3cm) -- +(67.5:3cm) -- +(112.5:3cm) -- +(157.5:3cm) --
+(202.5:3cm) -- +(247.5:3cm) -- +(292.5:3cm) -- +(337.5:3cm) -- cycle;
\end{tikzpicture}
\end{center}
\end{example}

Let $\Phi$ be the root system of $W$, realized in the real vector space $\mathbf E$ with basis $\{\alpha_1, \dots , \alpha_n \}$ with the symmetric bilinear form $B$ defined by
\[ B(\alpha_i, \alpha_j)= - \cos (\pi/m_{ij})  \text{ for } 1 \le i,j \le n.\] For each $i \in \{ 1, \dots , n \}$, define the action of $s_i$ on $\mathbf E$ by 
\[ s_i (\lambda) = \lambda -2B(\lambda, \alpha_i) \alpha_i , \quad \lambda \in \mathbf E, \]
and extend it to the action of $W$ on $\mathbf E$. Then each root $\alpha \in \Phi$ determines a reflection $s_\alpha \in W$.  (See \cite{Hu} for more details.)

\begin{definition} \label{def-rroot}
A positive root $\alpha \in \Phi$ of $W$ is called {\em rigid} if the corresponding reflection $s_\alpha \in W$ is rigid. 
\end{definition}

\begin{example}
In Example \ref{exa-1}, we obtained the rigid reflection $$(s_3s_2s_1)^4s_2s_3s_2s_1s_2s_3s_2s_3s_2s_1s_2s_3s_2(s_1s_2s_3)^4.$$ It give rises to a rigid root  
\begin{equation*} 1662490\alpha_1+4352663\alpha_2+11395212\alpha_3=(s_3s_2s_1)^4s_2s_3s_2s_1s_2s_3\alpha_2. \end{equation*}
\end{example}

 Fix a positive integer $m\geq 2$. As in the introduction, we set
 $$W(m)=\langle s_1,s_2,s_3 \ : \   s_1^2=s_2^2=s_3^2=(s_1s_2)^m=(s_2s_3)^m=e \rangle.
 $$ Note that we put, in particular, $m_{13}=m_{31}=\infty$.
Let  $\mathcal H(m)$ be the rank 2 hyperbolic Kac--Moody algebra associated with the Cartan matrix ${\tiny \begin{pmatrix}  2&-m\\-m&2\end{pmatrix}}$. We denote an element of the root lattice of $\mathcal H(m)$ by $[a,b]$, $a, b \in \mathbb Z$, where $[1,0]$ and $[0,1]$ are the positive simple roots. 
A root $[a,b]$ of $\mathcal H(m)$ is called {\em reduced} if $\gcd(a,b)=1$ and $ab\neq 0$. 
One can see that every non-simple real root is reduced.

Let $\mathcal P^+=\{ [a,b]\,:\, a, b \in \mathbb Z_{> 0}, \ \gcd(a,b)=1 \}$.
For every $[a,b] \in \mathcal P^+$, let $\eta([a,b])$ be the line segment  from $(0,0)$ to $(a,b)$ on the universal cover of the torus, which has no self-intersections.   Write   
$s([a,b]):=s(\eta([a,b]))\in W(m)$ for the corresponding rigid reflection.
For example, we have $s([5,3])=s_2s_3s_2s_1s_2s_3s_2s_3s_2s_1s_2s_3s_2$ as one can check in the following picture.

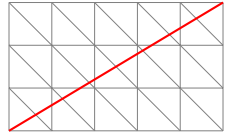
\begin{center}
\begin{tikzpicture}[scale=0.2mm]
\draw [help lines] (0,0) grid (5,3);
\draw [help lines] (0,1)--(1,0);
\draw [help lines] (0,2)--(2,0);
\draw [help lines] (0,3)--(3,0);
\draw [help lines] (1,3)--(4,0);
\draw [help lines] (2,3)--(5,0);
\draw [help lines] (3,3)--(5,1);
\draw [help lines] (4,3)--(5,2);
\draw [thick,red] (0,0)--(5,3);  
\end{tikzpicture}
\end{center}

With these definitions, we now state the main theorem of this paper.

\begin{thm} \label{thm-main}
The function, $[a,b] \mapsto s([a,b])$, is an onto function from the set of reduced positive roots of $\mathcal H(m)$ to  the set of rigid reflections of $W(m)$. 
\end{thm}

Equivalently, if we let $\beta([a,b])$ be the rigid root determined by the rigid reflection $s([a,b])$, then the above theorem asserts that
 the function, $[a,b] \mapsto \beta([a,b])$, is an onto function from the set of reduced positive roots of $\mathcal H(m)$ to  the set of rigid roots of $W(m)$.

\medskip

A proof of Theorem \ref{thm-main} will be given in the next section. In the rest of this section we will present some examples.
Recall from \cite{Kac} that
\begin{equation} \label{eqn-root}
\text{
$[a,b]$ is a root of $\mathcal H(m)$ \quad if and only if \quad $a^2+b^2-mab \le 1$. }
\end{equation}
We will use this fact in the following example  without further mentioning it. 

\begin{example} \hfill

(1) Let $m=3$. Consider the rigid reflection $s([4,1])= s_2s_3s_2s_3s_2s_3s_2= s_2$ and its rigid root $\beta([4,1])=\alpha_2$. The point $[4,1]$ is not a root of $\mathcal H(3)$.  However, these are covered by the root  $[1,1]$ of $\mathcal H(3)$   since $s([1,1])=s_2$ and $\beta([1,1])=\alpha_2$.

One can also check 
$s([30,11])=s_2s_3s_2s_1s_2s_3s_2 =s([3,2])$ and $\beta([30,11])= \alpha_1+3\alpha_2 +3\alpha_3 =\beta([3,2])$.
Here $[30,11]$ is not a root of $\mathcal H(3)$, whereas $[3,2]$ is.

(2) Now let $m=4$. Then we have 
\begin{align*} s([5,2])&=s([13,2]) = s_2s_3 s_2s_3 s_2s_1s_2s_3 s_2s_3s_2 \ \text{ and } \\ \beta([5,2])&=\beta([13,2]) = \alpha_1 + 3 \sqrt 2 \alpha_2 + 6 \alpha_3 .\end{align*}
Here $[13,2]$ is not a root of $\mathcal H(4)$, but $[5,2]$ is a root.

(3) For a general $m$, let $x= 2 \cos(\pi/m)$.  
Then $s([5,3])=s_2s_3s_2s_1s_2s_3s_2s_3s_2s_1s_2s_3s_2$
and \[ \beta([5,3]) = (x^3 + x) \alpha_1+(x^6 + 3x^4 + 2x^2 - 1)\alpha_2+(x^5 + 3x^3 + 2x)\alpha_3 .\]

\end{example}

\begin{example} \label{exm-2}
Assume that $m=2$. Then the Kac--Moody algebra $\mathcal H(2)$ is the affine Lie algebra $\widehat{\mathfrak{sl}}_2$ and its set of reduced positive roots is given by 
\[ \{ [n,n+1], \ [n+1,n], \ [1,1] \ : \ n \ge 1 \}. \] On the other hand, since $s_2$ commutes with $s_1$ and $s_3$ in $W(2)$, we see that the set of rigid reflections in $W(2)$ is
\[ \{ s_1(s_3s_1)^{n-1}, \ s_3(s_1s_3)^{n-1}, \ s_2 \ : \ n \ge 1 \} , \] 
and that the set of rigid roots of $W(2)$ is
\[ \{ n \alpha_1+ (n-1) \alpha_3, \ (n-1) \alpha_1+ n \alpha_3, \ \alpha_2, \ :\ n \ge 1 \} . \] Applying the maps $s(\cdot)$ and $\beta(\cdot)$ to the set of reduced positive roots, we obtain, for $n \ge 1$,  \begin{align*} s([n,n+1]) &= s_1(s_3s_1)^{n-1} ,  & s([n+1,1])&= s_3(s_1s_3)^{n-1} , & s([1,1]) &= s_2,  \\ \beta([n,n+1]) &= n \alpha_1+ (n-1) \alpha_3 ,  & \beta([n+1,1])&= (n-1) \alpha_1+ n \alpha_3 ,  & \beta([1,1]) &= \alpha_2 . \end{align*}
Therefore, the maps are clearly bijections, and Conjecture \ref{conj} is verified in this case $m=2$.

\end{example}

\section{Proof} \label{proof}

In this section, we prove Theorem \ref{thm-main}. The last lemma (Lemma \ref{lem-page}) enables us to focus only on the positive root lattice $\mathcal P^+$. Lemma \ref{lem-sequiv} shows that we can use a certain transformation  to preserve rigid roots. Lemmas \ref{lem-ab} and \ref{lem-ineq} guarantee this transformation to work inductively, and the inductive algorithm is given in Proposition \ref{pro-pr}. We explain the algorithm in Example \ref{exm-3}.

\medskip

Define a sequence $F_n$ recursively by $F_0=0, F_1=1$, and $F_n=mF_{n-1}-F_{n-2}$.
Define another sequence $E_n$ by $E_0=E_1=1$ and $E_n=mE_{n-1}-E_{n-2}$.

\begin{lem} \label{lem-ab}
Assume that $[a,b]\in \mathcal P^+$ is not a root of $\mathcal H(m)$, and that \[ \frac {F_{n-1}}{F_{n}} < \frac b a < \frac {F_{n}}{F_{n+1}} .\]
We have either 
\begin{align*}
& [a,b]-m(-F_{n-1}a + F_{n}b)[F_{n}, F_{n-1}]  \in \mathcal P^+, \quad \text{ or } \\ & [a,b]+m(-F_{n}a + F_{n+1}b)[F_{n+1}, F_{n}]  \in \mathcal P^+ . 
\end{align*}
\end{lem}

\begin{proof}
Let $u=F_{n-1}, v=F_{n}$ and $w=F_{n+1}$ for convenience. Then we have
\[ mv=u+w \quad \text{ and } \quad v^2=1+uw. \]
We want to show
\begin{align*} [c,d] &:= [(1+muv)a -mv^2b, mu^2a+(1-muv)b ] \in \mathcal P^+, \quad \text{ or } \\ [e,f]  & :=[(1-mvw)a+ mw^2b,  -mv^2a+( 1+mvw)b] \in \mathcal P^+ .\end{align*}
Note that the matrices
\[ \begin{pmatrix} 1+muv & -mv^2 \\ mu^2 & 1-muv \end{pmatrix} \quad \text{ and } \quad \begin{pmatrix} 1-mvw & mw^2 \\ -mv^2 & 1+mvw \end{pmatrix} \]
have determinant $1$. Since $\gcd(a,b)=1$, we have $\gcd(c,d)=1$ and $\gcd(e,f)=1$.

Assume that $c<0$. We claim that $e>0$ and $f>0$.  The condition $c<0$ is equivalent to $(1+muv)a <mv^2b$, and the conditions $e>0$ and $f>0$ are equivalent to $mw^2b>(mvw-1)a$ and $(1+mvw)b>mv^2 a$, respectively. Thus we have only to prove
\begin{equation} \label{eqn-in-1} \frac {1+muv}{mv^2} > \frac {mvw-1}{mw^2} \quad \text{ and } \quad \frac {1+muv}{mv^2} > \frac {mv^2}{1+mvw} .\end{equation} 

The first inequality in \eqref{eqn-in-1} is equivalent to
\[ (1+muv) w^2 > (mvw-1) v^2 .\] Using $mv=u+w$ and $v^2=1+uw$, we obtain
\begin{align*}
\mathrm{LHS} &= (1+u(u+w))w^2 = w^2+u^2w^2+uw^m, \\
\mathrm{RHS} &= ((u+w)w-1)(1+uw)= w^2-1+u^2w^2+uw^m.
\end{align*}
It establishes the first inequality.

The second inequality in \eqref{eqn-in-1} is equivalent to
\[ (1+muv)(1+mvw) > m^2v^4 ,\]
and we have
\begin{align*}
\mathrm{LHS} &= 1+mv(u+w)+m^2uwv^2=1+m^2v^2+m^2uwv^2=1+m^2v^2(1+uw)\\& =1+ m^2v^4 > m^2v^4 =\mathrm{RHS}.
\end{align*}
It establishes the second inequality.

Now assume that $d<0$. We again claim that $e>0$ and $f>0$.  The condition $d<0$ is equivalent to $mu^2a <(muv-1)b$.  Thus we have only to prove
\begin{equation} \label{eqn-in} \frac {mu^2}{muv-1} > \frac {mvw-1}{mw^2} \quad \text{ and } \quad \frac {mu^2}{muv-1} > \frac {mv^2}{1+mvw} ,\end{equation} 
which can be checked similarly to the case $c<0$. 

Thus we have shown that if $[c,d] \not \in \mathcal P^+$ then $[e,f] \in \mathcal P^+$, and it completes the proof.
\end{proof}

Let $(a_1, a_2)$ be a pair of positive integers with $a_1\geq a_2$. 
A \emph{maximal Dyck path} of type $a_1\times a_2$, denoted by $\mathcal{D}^{a_1\times a_2}$, is a lattice path
from $(0, 0)$  to $(a_1,a_2)$ that is as close as possible to the diagonal joining $(0,0)$ and $(a_1,a_2)$ without ever going above it. Assign $s_2s_3\in W(m)$ to each horizontal edge of $\mathcal{D}^{a_1\times a_2}$, and $s_2s_1\in W(m)$ to each vertical edge. Read these elements in the order of edges along $\mathcal{D}^{a_1\times a_2}$, then we get a product of copies of $s_2s_3$ and $s_2s_1$. Denote the product by $s^{a_1\times a_2}$.

Let $\sigma_1$ and $\sigma_2$ be the simple reflections of $\mathcal H(m)$ associated to the simple roots $[1,0]$ and $[0,1]$, respectively. Then they act on $[a,b] \in \mathbb Z^2$ in the usual way by
\[ \sigma_1 [a,b]=[-a+mb, b] \quad \text{ and } \quad \sigma_2[a,b]=[a,-b+ma] .\]

\begin{lem}\label{sab} Assume that $[a,b] \in \mathcal P^+$ with $a\geq b$, and write $[c,d]=\sigma_1\sigma_2[a,b]$. Then we have
\begin{enumerate}
\item [(1)] $s([a,b])=s_3s_2 s^{a\times b} s_1$;
 
\item [(2)] $s_1s_3s_2s^{a\times b} s_2s_3s_1=s^{c \times d}$;

\item [(3)] $s_3s_2s_1 s([a,b]) s_1s_2s_3 = s([c,d])$. 
\end{enumerate}
 \end{lem}

\begin{proof}
(1) Straightforward. 

(2) Given $\mathcal D^{a \times b}$, we replace a horizontal step, which is followed by another horizontal step, with $\mathcal D^{(m^2-1) \times m}$ and a two-step path with horizontal step and an immediate vertical step with $\mathcal D^{(m^2-m-1) \times (m-1)}$. 
Then the resulting path is $\mathcal D^{c\times d}$. This transformation of Dyck paths can also be obtained from a sequence of Dyck path mutations considered in \cite[Section 3]{LLZ}.

For example, when $m=3$, we have
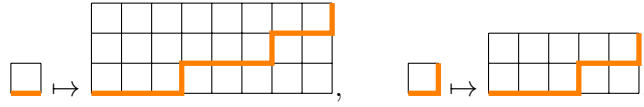
\[ \begin{tikzpicture}[scale=0.4]
 \draw (0,0) grid (1,1);
\draw [orange, line width=2] (0,0) -- (1,0);
\end{tikzpicture}
\mapsto
\begin{tikzpicture}[scale=0.4]
 \draw (0,0) grid (8,3);
\draw [orange, line width=2] (0,0) -- (3,0)--(3,1)--(6,1)--(6,2)--(8,2)--(8,3);
\end{tikzpicture},
\qquad
 \begin{tikzpicture}[scale=0.4]
 \draw (0,0) grid (1,1);
\draw [orange, line width=2] (0,0) --(1,0)--(1,1);
\end{tikzpicture}
\mapsto
\begin{tikzpicture}[scale=0.4]
 \draw (0,0) grid (5,2);
\draw [orange, line width=2] (0,0) -- (3,0)--(3,1)--(5,1)--(5,2);
\end{tikzpicture} , \]
and obtain $\mathcal D^{13 \times 5}$ from $\mathcal D^{2 \times 1}$ through this transformation:
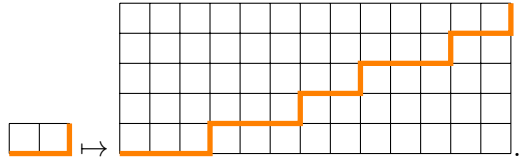
\[ \begin{tikzpicture}[scale=0.4]
 \draw   (0,0) grid (2,1);
\draw [orange, line width=2] (0,0) -- (2,0)--(2,1);
\end{tikzpicture}
\mapsto
\begin{tikzpicture}[scale=0.4]
 \draw (0,0) grid (13,5);
\draw [orange, line width=2] (0,0) -- (3,0)--(3,1)--(6,1)--(6,2)--(8,2)--(8,3)--(11,3)--(11,4)--(13,4)--(13,5);
\end{tikzpicture}.\]

We consider the associated Coxeter group elements
\begin{align*} s^h&:=s^{(m^2-1) \times m} = s([m,1])^{m-1}s([m-1,1])=s_1s_2s_3s_1 \quad \text{ and } \\  s^{hv}&:=s^{(m^2-m-1) \times (m-1)} = s([m,1])^{m-2}s([m-1,1]) =s_1s_2s_1s_2s_3s_1,
\end{align*}
and obtain 
\begin{align*}
 s_2s_3s_1 s^{h} s_1s_3s_2 &=(s_2s_3s_1)(s_1s_2s_3s_1)(s_1s_3s_2) = s_2s_3,\\
 s_2s_3s_1 s^{hv} s_1s_3s_2 &=(s_2s_3s_1)(s_1s_2s_1s_2s_3s_1)(s_1s_3s_2) = (s_2s_3)(s_2s_1).
 \end{align*}
This proves the assertion.

(3) This is an immediate consequence of the parts (1) and (2).
\end{proof}

\begin{lem}\label{sFnFn1} We have the following formulas.
\begin{enumerate}
\item[(a)] \qquad 
$s^{F_2\times F_1} =s_2s_1 \quad \text{ and } \quad s^{E_2\times E_1}=s_3s_1$. 
\begin{align*}
s^{F_{n}\times F_{n-1}}&=\left\{\begin{array}{ll}s_1(s_3s_2s_1)^{(n-3)/2}s_2s_3(s_1s_2s_3)^{(n-3)/2}s_1, & \qquad \text{ for }n\geq3\text{ odd};\\
 s_1(s_3s_2s_1)^{(n-4)/2}s_3s_1s_2s_3(s_1s_2s_3)^{(n-4)/2}s_1, & \qquad \text{ for }n\geq4\text{ even.}\end{array} \right .   \\ 
s^{E_{n}\times E_{n-1}} &=\left\{\begin{array}{ll}s_1(s_3s_2s_1)^{(n-3)/2}s_2s_1s_2s_3(s_1s_2s_3)^{(n-3)/2}s_1, &\ \text{ for }n\geq3\text{ odd};\\
 s_1(s_3s_2s_1)^{(n-4)/2}s_3s_2s_3s_1s_2s_3(s_1s_2s_3)^{(n-4)/2}s_1, &\ \text{ for }n\geq4\text{ even.} \end{array} \right . 
\end{align*}
\item[(b)] \qquad $ (s_3s_2s_1 s([F_n,F_{n-1}]))^m=(s_1s_2s_3 s([F_n,F_{n-1}]))^m=e  \qquad \text{ for all } n \ge 1 . $

\end{enumerate}
\end{lem} 
\begin{proof}
(a) Let $\mathcal{F}_n:=\mathcal{D}^{F_{n}\times F_{n-1}}$ and $\mathcal{E}_n:=\mathcal{D}^{E_{n}\times E_{n-1}}$. 
We use induction on $n$. It is easy to check the base cases. Suppose $n\geq 3$. Then the Dyck path $\mathcal{F}_n$ consists of $m-1$ copies of $\mathcal{F}_{n-1}$ followed by one copy of $\mathcal{E}_{n-1}$. This is because $((m-1)F_{n-1},(m-1)F_{n-2})$ is below the diagonal, and Pick's theorem implies that there is no integral point in the interior of the triangle formed by $(0,0),(F_n,F_{n-1})$, and $((m-1)F_{n-1},(m-1)F_{n-2})$. Similarly, the Dyck path $\mathcal{E}_n$ consists of $m-2$ copies of $\mathcal{F}_{n-1}$ followed by one copy of $\mathcal{E}_{n-1}$.  It is straightforward to check the induction process as follows.

Suppose that $n$ is odd. From the induction hypothesis, we have  \begin{align*} s^{F_{n}\times F_{n-1}} & =s_1(s_3s_2s_1)^{(n-3)/2}s_2s_3(s_1s_2s_3)^{(n-3)/2}s_1,\ \text{ and } \\ s^{E_{n}\times E_{n-1}} & =s_1(s_3s_2s_1)^{(n-3)/2}s_2s_1s_2s_3(s_1s_2s_3)^{(n-3)/2}s_1.
\end{align*} 
Then $$\aligned
&s^{F_{n+1}\times F_{n}}\\
&=(s_1(s_3s_2s_1)^{(n-3)/2}s_2s_3(s_1s_2s_3)^{(n-3)/2}s_1)^{m-1}(s_1(s_3s_2s_1)^{(n-3)/2}s_2s_1s_2s_3(s_1s_2s_3)^{(n-3)/2}s_1)\\
&=s_1(s_3s_2s_1)^{(n-3)/2}(s_2s_3)^{m-1}s_2s_1s_2s_3(s_1s_2s_3)^{(n-3)/2}s_1\\
&=s_1(s_3s_2s_1)^{(n-3)/2}s_3s_1s_2s_3(s_1s_2s_3)^{(n-3)/2}s_1,
\endaligned$$ and
$$\aligned
&s^{E_{n+1}\times E_{n}}\\
&=(s_1(s_3s_2s_1)^{(n-3)/2}s_2s_3(s_1s_2s_3)^{(n-3)/2}s_1)^{m-2}(s_1(s_3s_2s_1)^{(n-3)/2}s_2s_1s_2s_3(s_1s_2s_3)^{(n-3)/2}s_1)\\
&=s_1(s_3s_2s_1)^{(n-3)/2}(s_2s_3)^{m-2}s_2s_1s_2s_3(s_1s_2s_3)^{(n-3)/2}s_1\\
&=s_1(s_3s_2s_1)^{(n-3)/2}s_3s_2s_3s_1s_2s_3(s_1s_2s_3)^{(n-3)/2}s_1.
\endaligned$$
The other case can be similarly proved.

(b) Each of $s_3s_2s_1 s([F_n,F_{n-1}])=s_3s_2s_1 s_3s_2 s^{F_n\times F_{n-1}}s_1$ and $s_1s_2s_3 s([F_n,F_{n-1}])=s_1s^{F_n\times F_{n-1}}s_1$   is a conjugate of one of $s_1s_2, s_2s_1, s_2s_3$, or $s_3s_2$, which implies the statement. 
\end{proof}

\begin{remark} \label{rmk-1}
Notice that  $s_3s_2s_1$ can be considered as a curve going around below an integral point  and $s_1s_2s_3$ going around above an integral point. See the illustrations below.
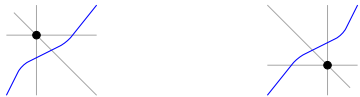
\begin{center}
\begin{tikzpicture}[scale=1]
\draw[step=1, gray, opacity=0.7, very thin]  (-0.3,1.3) -- (0.8,0.2);
\draw[step=1, gray, opacity=0.7, very thin]  (-0.4,1) -- (0.8,1);
\draw[step=1, gray, opacity=0.7, very thin]  (0,1.4) -- (0,0.2);
\draw [fill, black] (0,1) circle [radius=0.05];
\draw[rounded corners=1mm, blue] (-0.4,0.2)--(-0.2,0.6)--(0.4, 0.9)--(0.8,1.4);
\end{tikzpicture}
\hspace{2 cm}
\begin{tikzpicture}[scale=1]
\draw[step=1, gray, opacity=0.7, very thin]  (0.3,-1.3) -- (-0.8,-0.2);
\draw[step=1, gray, opacity=0.7, very thin]  (0.4,-1) -- (-0.8,-1);
\draw[step=1, gray, opacity=0.7, very thin]  (0,-1.4) -- (0,-0.2);
\draw [fill, black] (0,-1) circle [radius=0.05];
\draw[rounded corners=1mm, blue] (0.4,-0.2)--(0.2,-0.6)--(-0.4, -0.9)--(-0.8,-1.4);
\end{tikzpicture}
\end{center}
This fact will be used frequently in the proof of the following lemma.
\end{remark}

\begin{lem} \label{lem-sequiv}
For a fixed positive integer $n$ and $[a,b] \in \mathcal P^+$ with $a>b$, let $\kappa= -F_{n-1}a+F_{n}b$.  Assume that $a':=a-m|\kappa|F_n>0$ and $b':=b-m|\kappa|F_{n-1}>0$. Then
$$
s([a',b']) = s([a,b]).
$$
\end{lem}
\begin{proof}
If $\kappa=0$ then trivial. Here we give a proof for the case of $\kappa>0$, as the $\kappa<0$ case is similar. We need to be able to locate the integral points inside the triangle, say $T$, formed by $(0,0)$, $(a,b)$ and $(a',b')$.  Pick's theorem implies that there are exactly $m{\kappa\choose 2}$  integral points  in the interior of $T$.     For each $i\in\{1,...,\kappa\}$, let $L_i\subset \mathbb{R}^2$ be the line segment from $(\frac{i}{\kappa}a',\frac{i}{\kappa}b')$ to $(\frac{i}{\kappa}a,\frac{i}{\kappa}b)$. For technical simplicity, assume that $L_\kappa$ is open ended at $(a,b)$ so that $L_\kappa$ contains exactly $m\kappa$ integral points.
Also note that $\kappa >0$ implies  $\frac b a < \frac {b'}{a'}$.  

First, assume $\kappa =1$. Then there is no integral point inside $T$, and by  Lemma \ref{sFnFn1}\,(b) and Remark \ref{rmk-1}, we have
\[ s([a,b])= s([a',b']) (s_3s_2s_1 s([F_n, F_{n-1}]))^m =s([a',b']). \]
For example, consider the case $m=3$, $n=2$ and $(F_2,F_1)=(3,1)$. If $[a,b]=[11,4]$, then $\kappa=1$ and $[a',b']=[2,1]$. Then we have
\[ s([11,4])= s_2s_3s_2(s_3s_2s_1s_2s_3s_2s_3s_2)^3 =s_2s_3s_2=s([2,1]).\]

Now assume $\kappa >1$. Let $g_a=\gcd(a,F_n)$ and $g_b=\gcd(b,F_{n-1})$.  
There exists an integer $x\not\equiv 0\ (\text{mod }\kappa)$ such that  $\frac{a}{g_a}-x\frac{F_{n}}{g_a}$ is divisible by $-F_{n-1}\frac{a}{g_a}+\frac{F_{n}}{g_a}b$, and there exists an integer $y\not\equiv 0\ (\text{mod }\kappa)$ such that  $\frac{b}{g_b}-y\frac{F_{n-1}}{g_b}$ is divisible by $-\frac{F_{n-1}}{g_b}a+F_{n}\frac{b}{g_b}$. We want to show that there is an integer $z\not\equiv 0 \ (\text{mod }\kappa)$ satisfying both conditions simultaneously, i.e., $x=y$ can be made. 

Note that $-F_{n-1}(a-xF_n)+ F_n(b-yF_{n-1})$ (hence $(x-y)F_{n-1}F_n$) is divisible by  $\kappa$.  Since $\gcd(F_{n-1},F_{n})=1$, we must have that $x-y$ is divisible by $-\frac{F_{n-1}}{g_b}\frac{a}{g_a}+\frac{F_{n}}{g_a}\frac{b}{g_b}$. Let $c$ be the integer such that 
$$
x-y=c\left(-\frac{F_{n-1}}{g_b}\frac{a}{g_a}+\frac{F_{n}}{g_a}\frac{b}{g_b}\right).
$$
By Bezout's theorem, $\gcd(g_a,g_b)=1$ implies that there exist two integers $c_1$ and $c_2$ such that $c_1g_a -c_2g_b=-c$, equivalently
$$\left(x+c_1\left(-\frac{F_{n-1}}{g_b}a+F_{n}\frac{b}{g_b}\right)\right)- \left(y+ c_2\left(-F_{n-1}\frac{a}{g_a}+\frac{F_{n}}{g_a}b\right)\right)=0.$$
Hence there is an integer $z\not\equiv 0 \ (\text{mod }\kappa)$ such that $z\equiv x+ c_1\frac{\kappa}{g_b} \equiv y+c_2\frac{\kappa}{g_a}(\text{mod }\kappa)$ and     both $a-zF_n$ and $b-zF_{n-1}$ are divisible by $\kappa$. 

This implies that there are at least $mi$ integral points, say $P_{i,1},...,P_{i,mi}$ (from the left), on $L_i$. In particular, there are at least $m{\kappa\choose 2}$ integral points in the interior of $T$, but we know that there are no more. For $i\in\{1,...,\kappa\}$, let $M_i$ be an admissible curve which starts at $(0,0)$, goes below $P_{i,1},...,P_{i,mi}$ but above $P_{i-1,1},...,P_{i-1,m(i-1)}$, and ends at $(a,b)$. It would be useful to give two different names to $M_i$ by letting 
$M_i^{-}$ (resp. $M_i^+$) be a curve (isotopic to $M_i$) sufficiently close to $P_{i-1,1},...,P_{i-1,m(i-1)}$ (resp. $P_{i,1},...,P_{i,mi}$). Note that $s(M_i)=s(M_i^+)=s(M_i^-)$.

For $i\in\{1,...,\kappa-1\}$, let $S_i$ be the line segment from $P_{i-1,1}$ to $P_{i,1}$ (where $P_{0,1}=(0,0)$), and  let $T_i$ be the line segment from $P_{i,mi+1}$ to $P_{i+1,m(i+1)+1}$, where $P_{i,mi+1}$ is the integral point that makes $P_{i,mi}$ become the midpoint between $P_{i,mi-1}$ and $P_{i,mi+1}$. 

Then, by  Lemma \ref{sFnFn1}\,(b) and Remark \ref{rmk-1},  we have 
$$\aligned s([a,b]) &= s( M_1 )=s( M_1^+ )=s(S_1) (s_3s_2s_1s([F_n,F_{n-1}]))^m s_1s_2s_3s(T_1)\cdots s_1s_2s_3s(T_{\kappa-1})\\
  &=s(S_1) (s_1s_2s_3s([F_n,F_{n-1}]))^m s_1s_2s_3s(T_1)\cdots s_1s_2s_3s(T_{\kappa-1}) = s(M_2^-)=s(M_2)\\ 
  &=s(M_2^+)=s(S_1)s_1s_2s_3s(S_2) (s_3s_2s_1s([F_n,F_{n-1}]))^{2m} s_1s_2s_3s(T_2)\cdots s_1s_2s_3s(T_{\kappa-1})\\
  &=\cdots=s(S_1)s_1s_2s_3s(S_2)\cdots s_1s_2s_3s(S_{\kappa-1}) (s_1s_2s_3s([F_n,F_{n-1}]))^{(\kappa-1)m} s_1s_2s_3s(T_{\kappa-1})\\
  &=s(M_{\kappa}^-)=s(M_{\kappa}^+)=s([a',b'])(s_3s_2s_1s([F_n,F_{n-1}]))^{\kappa m}= s([a',b']). \endaligned$$

\begin{figure}
  \centering
\begin{tikzpicture}[scale=0.4]
\draw[red] (0,0)--(3,8);
\draw[orange] (0,0)--(2,3)--(4,6); \draw[orange] (11,6)--(22,12)--(30,17);
\draw[rounded corners=1mm, blue] (0,0)--(2,2.5)--(8.5,4.8)--(21.8,12)--(30,17);
\draw[rounded corners=1mm, blue] (0,0)--(1.6,3)--(4,5)--(19,10.8)--(30,17);
\draw[rounded corners=1mm, blue] (0,0)--(1.4,3)--(3.6,6)--(6,8)--(24,14)--(30,17);
 \draw [fill] (6,9) circle [radius=0.1];\draw [fill] (3,8) circle [radius=0.1];\draw (3,8) node[anchor=east]  {\tiny$P_{3,1}$}; \draw [fill] (9,10) circle [radius=0.1];\draw [fill] (12,11) circle [radius=0.1];\draw [fill] (15,12) circle [radius=0.1]; \draw [fill] (18,13) circle [radius=0.1]; \draw [fill] (21,14) circle [radius=0.1];
 \draw [fill] (24,15) circle [radius=0.1]; \draw [fill] (27,16) circle [radius=0.1];
 \draw [fill] (30,17) circle [radius=0.1];\draw (30,17) node[anchor=west]  {\tiny$P_{3,10}$};
\draw [fill] (2,3) circle [radius=0.1];
\draw (2,3) node[anchor=west]  {\tiny$P_{1,1}$};
\draw [fill] (5,4) circle [radius=0.1];
\draw [fill] (8,5) circle [radius=0.1]; \draw [fill] (11,6) circle [radius=0.1];\draw (11,6) node[anchor=west]  {\tiny$P_{1,4}$};
 \draw [fill] (4,6) circle [radius=0.1];\draw (4,6) node[anchor=west]  {\tiny$P_{2,1}$};
 \draw [fill] (7,7) circle [radius=0.1];\draw [fill] (10,8) circle [radius=0.1];\draw [fill] (13,9) circle [radius=0.1];\draw [fill] (16,10) circle [radius=0.1];
 \draw [fill] (19,11) circle [radius=0.1]; \draw [fill] (22,12) circle [radius=0.1]; \draw (22,12) node[anchor=west]  {\tiny$P_{2,7}$};
\end{tikzpicture}
  \caption{A picture illustrating the case of $m=\kappa=3$, $[a,b]=[48,17]$ and $[a',b']=[21,8]$. The three blue-colored curves represent $M_1,M_2$, and $M_3$ (from the bottom). The four orange-colored line segments are $S_1,S_2,T_1$, and $T_2$.}
\end{figure}
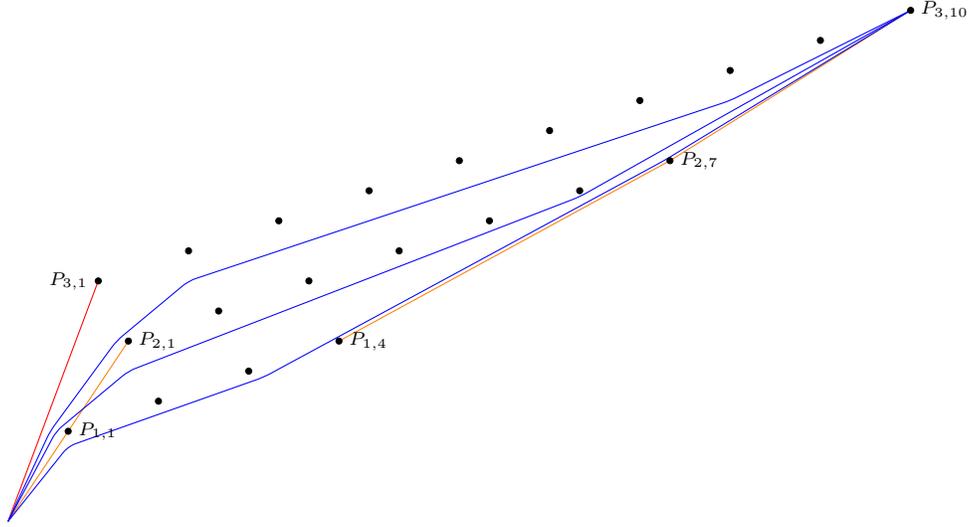

\end{proof}

When $n=0$ in the above lemma, we have $\kappa = b$ and obtain the following corollary.
\begin{cor} \label{cor-a}
Let $[a,b] \in \mathcal P^+$. Then, for $j \in \mathbb Z_{>0}$,
\[ s([a+jmb, b])=s([a,b]).\]
\end{cor}


For $[a,b]\in \mathcal P^+$, define 
\[ Q([a,b])=a^2+b^2-mab .\]

\begin{lem} \label{lem-ineq}
 Assume  that \[ \frac {F_{n-1}}{F_{n}} < \frac b a < \frac {F_{n}}{F_{n+1}}, \qquad [a,b]\in \mathcal P^+.\] Then, for any $j \in \mathbb Z_{>0}$, we have
\begin{align*} & Q \left ( [a,b]-mj(-F_{n-1}a + F_{n}b)[F_{n}, F_{n-1}] \right ) < Q([a,b]), \quad \text{ and }\\
& Q \left ( [a,b]+mj(-F_{n}a + F_{n+1}b)[F_{n+1}, F_{n}] \right ) < Q([a,b]). 
 \end{align*}
\end{lem}

\begin{proof}
Let $u=F_{n-1}, v=F_{n}$ and $w=F_{n+1}$ for convenience. First, set
\[ x=a+mjauv-mjbv^2 \quad \text{ and } \quad y=b+mjau^2-mjbuv .\]
We want to prove \[x^2 -mxy+y^2 < a^2 -m ab + b^2 . \]
We compute  
\begin{align} \label{eq-to}
& x^2 -mxy+y^2 -(a^2 -m ab + b^2 )  \\ &= 
mj(mju^4 - m^2ju^mv + mju^2v^2 - mu^2 + 2uv)a^2 \nonumber \\
& + mj(mju^2v^2 - m^2juv^m + mjv^4 - 2uv + mv^2)b^2 \nonumber  \\
&- 2mj(mju^mv - m^2ju^2v^2 + mjuv^m - u^2 + v^2)ab \nonumber  
\end{align}

Since  $ua<vb$, we multiply the identity \eqref{eq-to} by $v/mj$ and then replace $vb$ by $ua$ to obtain the inequality
\begin{align*}
\eqref{eq-to}\times ( v/mj ) &< (mju^4v - m^2ju^mv^2 + mju^2v^m - mu^2v + 2uv^2)a^2 \\
& + (mju^2v - m^2juv^2 + mjv^m - 2u + mv)u^2a^2 \\
&- 2(mju^mv - m^2ju^2v^2 + mjuv^m - u^2 + v^2)ua^2 =0.
\end{align*}
Thus we have \eqref{eq-to}$<0$ as desired.

Now let 
\[ x_1=a-mjavw+mjbw^2 \quad \text{ and } \quad y_1=b-mjav^2+mjbvw ,\]
and compute to obtain
\begin{align} \label{eq-too}
& x_1^2 -mx_1y_1+{y_1}^2 -(a^2 -m ab + b^2 )  \\ &= 
mj(mjv^4 - m^2jv^mw + mjv^2w^2 + mv^2 - 2vw)a^2 \nonumber \\
& + mj(mjv^2w^2 - m^2jvw^m + mjw^4 + 2vw - mw^2)b^2 \nonumber  \\
&- 2mj(mjv^mw - m^2jv^2w^2 + mjvw^m + v^2 - w^2)ab .\nonumber  
\end{align}

Since $w b <va$, we multiply the identity \eqref{eq-too} by $v/mj$ and then replace $va$ by $wb$  to obtain the inequality
\begin{align*}
\eqref{eq-too}\times ( v/mj ) &< (mjv^m - m^2jv^2w + mjvw^2 + mv - 2w)w^2b^2\\
&  + (mjv^mw^2 - m^2jv^2w^m + mjvw^4 + 2v^2w - mvw^2)b^2 \\
& - 2(mjv^mw - m^2jv^2w^2 + mjvw^m + v^2 - w^2)wb^2 =0.
\end{align*}
This implies \eqref{eq-too}$<0$ as desired.

\end{proof}

Combining the lemmas in this section, we obtain the following proposition, which is a main step toward the proof of Theorem \ref{thm-main}.
\begin{prop} \label{pro-pr}
Assume that $[a,b]\in \mathcal P^+$. Then there exists $[a_0, b_0] \in \mathcal P^+$ such that $[a_0,b_0]$ is a reduced root of $\mathcal H(m)$ and $s([a_0,b_0]) = s([a,b])$ or equivalently, $\beta([a_0,b_0]) = \beta([a,b])$.  
\end{prop}

\begin{proof}
If $[a,b]\in \mathcal P^+$ is a root of $\mathcal H(m)$, it is already reduced from the definition of $\mathcal P^+$ and we simply take $[a_0, b_0]=[a,b]$. Assume that $[a,b]$ is not a root of $\mathcal H(m)$. Without loss of generality, we may further assume that $a>b$. Then we have   $\frac {F_{n-1}}{F_{n}} < \frac b a < \frac {F_{n}}{F_{n+1}}$ for some $n \in \mathbb Z_{>0}$. By Lemma \ref{lem-ab}, we have 
either $[a,b]-m(-F_{n-1}a + F_{n}b)[F_{n}, F_{n-1}]  \in \mathcal P^+$ or $[a,b]+m(-F_{n}a + F_{n+1}b)[F_{n+1}, F_{n}]  \in \mathcal P^+$.
Put $[a',b']= [a,b]-m(-F_{n-1}a + F_{n}b)[F_{n}, F_{n-1}]$ or $[a',b']= [a,b]+m(-F_{n}a + F_{n+1}b)[F_{n+1}, F_{n}]$, so that $[a',b'] \in \mathcal P^+$.
Then  we have $s([a',b'])=s([a,b])$ and $Q([a',b'])<Q([a,b])$ by Lemmas \ref{lem-sequiv} and \ref{lem-ineq}, respectively. 

If $Q([a',b']) \le 1$ then $[a',b']$ is a positive reduced root of $\mathcal H(m)$ from \eqref{eqn-root} and we take $[a_0,b_0]=[a',b']$. If $Q([a',b']) > 1$ then $[a',b']$ is not a root of $\mathcal H(m)$ and we repeat the process by putting $a'$ and $b'$ to be new $a$ and $b$. Clearly, this process ends in a finite number of steps.
\end{proof}

\begin{example} \label{exm-3}
Assume that $m=3$, and consider $[487, 186]$. Since $Q([487,186])=19$, it is not a root of $\mathcal H(3)$. Note that $\frac {21}{55} < \frac {186}{487} < \frac {55}{144}$. We compute $3\times (-55\times 487 +144 \times 186) =-3$ and
\[ [487,186] - 3 \times [144,55] = [55,21] \in \mathcal P^+. \] Since $Q([55,21])=1$, $[55,21]$ is a real root of $\mathcal H(3)$ and the process ends here. Indeed, we have
$s([487,186])=s([55,21])$ and
\[ \beta([487,186])=\beta([55,21]) =6\alpha_1 + 8\alpha_2 +17 \alpha_3. \]

Now consider $[1789,683]$ with $Q([1789,683])=1349$. Since $\frac {8}{21} < \frac {683}{1789} < \frac {21}{55}$, we get
\[ [1789,683] +3(-21 \times 1789 + 55 \times 683) [55,21]= [1129,431] , \qquad Q([1129,431])=605 . \]
We continue to obtain
\begin{align*} &[1129,431] +3(-21 \times 1129 + 55 \times 431) [55,21]= [469,179] ,& & Q([469,179])=149,\\
&[469,179] -3(-8 \times 469 + 21 \times 179) [21,8]= [28,11] , & &Q([28,11])=-19. \end{align*}
Thus $[28,11]$ is an imaginary root of $\mathcal H(3)$, and we have \[ \beta([1789,683]) = \beta([28,11])=55 \alpha_1 + 55 \alpha_2 + 144 \alpha_3 . \]
\end{example}

\begin{lem} \label{lem-page}
Let $\eta$ be any non-self-crossing admissible curve with $\upsilon(\eta)\in\mathfrak{R}$. Then  there exists $[a,b]\in\mathcal{P}^+$ such that $s(\eta)=s([a,b])$.
\end{lem}
\begin{proof}
The curve $\eta$ is isotopic to a spiral (around the origin) followed by a line segment which is then followed by the opposite spiral (around the end point of $\eta$). Without loss of generality, we may assume that the first spiral goes around counterclockwise.  Then $s(\eta)$ can be written in one of the forms
\[ (s_3s_2s_1)^ns(\nu)(s_1s_2s_3)^n, \quad (s_3s_2s_1)^ns_3s(\nu)s_3(s_1s_2s_3)^n, \quad \text{and} \quad (s_3s_2s_1)^ns_3s_2s(\nu)s_2s_3(s_1s_2s_3)^n \] 
for some $n \ge 0$ and a line segment $\nu$. We will consider the case when $n=0$ and show that each of the forms is equal to $s([a,b])$ with $[a,b] \in \mathcal P^+$ and $a \ge b$. 
Then  we immediately obtain the statement for any $n>0$ by Lemma \ref{sab} (3) and an induction argument. Thus we only need to consider each of the reflections 
\[ s(\nu), \quad s_3s(\nu)s_3, \quad \text{and} \quad s_3s_2s(\nu)s_2s_3. \]

First, if $s(\eta)=s(\nu)$ with $\nu$ a line segment, then we have $s(\eta)=s([a,b])$ with  $a \ge b$, if necessary, by applying Corollary \ref{cor-a}. 

Next, if $s(\eta) = s_3s(\nu)s_3$, then $s(\nu)$ starts with the letter 1. 
Let ${D}(\eta)$ be the lattice path from $(0,0)$ to $(0,1)$ then to the end point of $\eta$ that goes North and West, and is closest to (but never crosses) the line segment $\nu$.  

Assign the element $s_2s_1\in W(m)$ to each vertical edge of  ${D}(\eta)$,  and $s_1s_2s_3s_1$ to each horizontal edge. Let $w(\eta)\in W(m)$ be the product of these elements obtained by reading them while traveling along $D(\eta)$. Then $s(\eta)=s_3s_2w(\eta)s_1$.

Let $(c,d)$ be the end point of $\eta$, and consider the line segment from $(0,0)$ to $(c+md,d)$.    Then $s(\eta)=s([c+md,d])$, and by applying Corollary \ref{cor-a} if necessary, we obtain $s(\eta)=s([a,b])$ with $a\ge b$. Since $c+d>0$, we have $[c+md, d] \in \mathcal P^+$. For example, the following figures illustrate the case $m=3$ and $(c,d)=(-2,3)$.
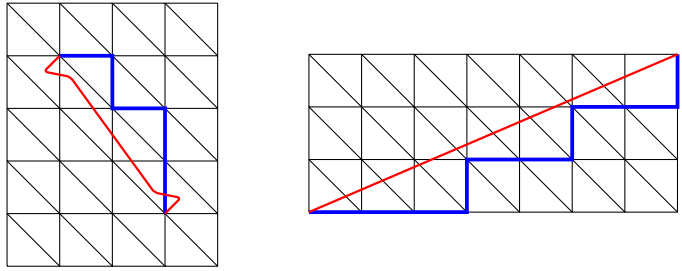
\begin{center}
\begin{tikzpicture}[scale=0.7]
\draw  (0,0) grid (4,5);
\draw (3,0)--(0,3);\draw (2,0)--(0,2);\draw (1,0)--(0,1);
\draw (4,0)--(0,4);\draw (4,1)--(0,5);\draw (4,2)--(1,5);\draw (4,3)--(2,5);\draw (4,4)--(3,5);
\draw[blue,line width= 0.5mm] (3,1)--(3,3)--(2,3)--(2,4)--(1,4);
\draw[red,rounded corners= 0.5 mm, line width= 0.3mm] (3,1)--(3.3,1.3)--(2.8,1.4)--(1.2,3.6)--(0.7,3.7)--(1,4);
\end{tikzpicture}$\hspace{30 pt}$
\raisebox{20pt}{\begin{tikzpicture}[scale=0.7]
\draw  (0,0) grid (7,3);
\draw (3,0)--(0,3);\draw (2,0)--(0,2);\draw (1,0)--(0,1);
\draw (4,0)--(1,3);\draw (5,0)--(2,3);\draw (6,0)--(3,3);\draw (7,0)--(4,3);\draw (7,1)--(5,3);
\draw[blue,line width= 0.5mm] (0,0)--(3,0)--(3,1)--(5,1)--(5,2)--(7,2)--(7,3);
\draw[red, line width= 0.3mm] (0,0)--(7,3);
\end{tikzpicture}}
\end{center}
Here we have $$\aligned s(\eta)
&=s_3s_2(s_2s_1)^2(s_1s_2s_3s_1)(s_2s_1)(s_1s_2s_3s_1)s_1=s_3s_1s_3s_1s_3\\
&=s_2s_3s_2s_3s_2s_1s_2s_3s_2s_3s_2s_1s_2s_3s_2s_3s_2 =s([7,3]).
\endaligned$$

Lastly, suppose that  $s(\eta) = s_3s_2s(\nu)s_2s_3$. Let $(c,d)$ be the end point of $\eta$. If $[c+md, d] \in \mathcal P^+$ then a similar argument to the previous case shows that $s(\eta) = s([c+md, d])$. Otherwise, we have $c+md <0$ and take a curve $\eta'$ such that 
\[ s(\eta')=s_3s_2 s(\nu') s_2 s_3 \]
where $\nu'$ is a line segment between $(c+md+\epsilon, d-\epsilon)$ and $(-\epsilon, \epsilon)$ for some small $\epsilon >0$. Then one can see that $s(\eta)=s(\eta')$. Repeating this process, we obtain a curve $\eta''$ whose end point is $(c'',d)$ with $c''+md>0$ and we are done by Corollary \ref{cor-a} as $s(\eta)=s([c''+md,d])$.
\end{proof}

\begin{proof}[Proof of Theorem \ref{thm-main}]
Assume that $s(\eta)$ is a rigid reflection of $W(m)$ given by a non-self-intersecting admissible curve $\eta$. By Lemma \ref{lem-page}, there exists $[a,b] \in \mathcal P^+$ such that $s(\eta)=s([a,b])$. Then by Proposition \ref{pro-pr} there exists a reduced positive root $[a_0,b_0]$ of $\mathcal H(m)$ such that 
\[ s(\eta)=s([a,b])=s([a_0,b_0]). \]
This completes the proof.
\end{proof}

\end{document}